%% file: Plan3.tex
\documentclass[12pt]{amsart}
\usepackage[T1]{fontenc}
\usepackage{amsmath}
\usepackage{amssymb}
\usepackage{amsthm}
\usepackage{pdfpages}
\usepackage{textcomp}
\usepackage{graphicx}
\usepackage{tikz}
\usetikzlibrary{patterns}
\usepackage{enumerate}

\newtheorem{theorem}{Theorem}[section]
\newtheorem{definitio}[theorem]{Definition}
\newenvironment{definition}{\begin{definitio} \rm }{\end{definitio}}
\newtheorem{lemma}[theorem]{Lemma}
\newtheorem{proposition}[theorem]{Proposition}
\newtheorem{corollary}[theorem]{Corollary}

\newtheorem*{theo}{Theorem}
\newtheorem*{prop}{Proposition}

\theoremstyle{remark}
\newtheorem*{remark}{Remark}
\newtheorem*{example}{Example}

\def\fl#1{\smash{\mathop{\hbox to 11mm{ \belowarrowfill\ }}\limits^{\scriptstyle{#1}}}}

\newcommand{\Z}{\mathbb{Z}}
\newcommand{\Q}{\mathbb{Q}}
\newcommand{\R}{\mathbb{R}}
\newcommand{\Zt}{\mathbb{Z}[t^{\pm1}]}
\newcommand{\Qt}{\mathbb{Q}[t^{\pm1}]}

\newcommand{\Lk}{\textrm{\textnormal{Lk}}}
\newcommand{\lk}{\textrm{\textnormal{lk}}}
\newcommand{\Al}{\mathfrak{A}}

\newcommand{\mir}[1]{\overline{#1}}
\newcommand{\Int}{\mathrm{Int}}
\newcommand{\Aut}{\mathrm{Aut}}
\newcommand{\Brib}{\textrm{\textit{\textbaht}}}
\newcommand{\Bpre}{\mathcal{B}}
\newcommand{\Drib}{\textrm{\textit{\DJ{}}}}
\newcommand{\cc}{\textrm{\textit{\textcentoldstyle}}}
\newcommand{\ccc}{\overline{\textrm{\textit{\textcentoldstyle}}}}

\newcommand{\tr}[1]{\hspace{-0.5ex}\ ^t\hspace{-0.25ex}#1}

\title[A Fox--Milnor theorem for the Alexander polynomial of $2$--knots]{A Fox--Milnor theorem for the Alexander polynomial of knotted $2$--spheres in $S^4$.}
\author{Delphine Moussard}
\address{Universit\'e{}~de Bourgogne Franche--Comt\'e, IMB, UMR 5584, 21000 Dijon, France } 
\email{delphine.moussard@u-bourgogne.fr}
\author{Emmanuel Wagner}
\address{Universit\'e{}~de Bourgogne Franche--Comt\'e, IMB, UMR 5584, 21000 Dijon, France } 
\email{emmanuel.wagner@u-bourgogne.fr}

\begin{document}

\begin{abstract}
For knots in $S^3$, it is well-known that the Alexander polynomial of a ribbon knot factorizes as $f(t)f(t^{-1})$ for some polynomial 
$f(t)$. By contrast, the Alexander polynomial of a ribbon $2$--knot is not even symmetric in general. Via an alternative notion of ribbon $2$--knots, 
we give a topological condition on a $2$--knot that implies the factorization of the Alexander polynomial. 
\end{abstract}

\maketitle

\section{Introduction}

The class of ribbon embeddings turns out to play a crucial role in low-dimensional topology. They first appeared in the work of Fox \cite{Fox} 
who investigated the problem of determining when a knot in the $3$--sphere bounds a smooth disk in the $4$--ball. Such a knot is called a slice knot. 
It is the case in particular when the knot bounds an immersed disk in the $3$--sphere with specific self-intersections, namely ribbon singularities, 
see Figure~\ref{figribbon}.
\newcommand{\hor}[1]{\draw[color=black!0,pattern = horizontal lines] #1;}
\newcommand{\neup}[1]{\draw[color=white,fill=white] #1; \draw[pattern = north east lines] #1;}
\newcommand{\horup}[1]{\draw[color=white,fill=white] #1; \draw[color=black!0,pattern = horizontal lines] #1;}
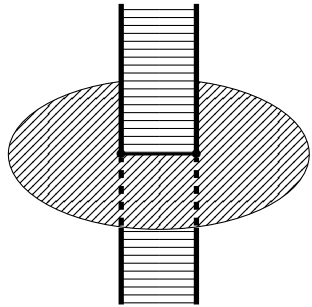
\begin{figure}[htb] 
\begin{center}
\begin{tikzpicture}
  \hor{(0.5,-2) -- (0.5,0) -- (-0.5,0) -- (-0.5,-2)}
  \neup{(0,0) circle (2 and 1)}
  \horup{(-0.5,0) -- (-0.5,2) -- (0.5,2) -- (0.5,0)}
  \foreach \x in {-0.5,0.5} {
  \draw[line width=2pt] (\x,0) -- (\x,2);
  \draw[line width=2pt] (\x,-0.98) -- (\x,-2);
  \draw[line width=2pt,dashed] (\x,-0.95) -- (\x,-0.05);}
  \draw[line width=1pt] (-0.5,0) node {$\scriptstyle{\bullet}$} -- (0.5,0) node {$\scriptstyle{\bullet}$};
\end{tikzpicture}
\end{center}
\caption{A ribbon singularity.} \label{figribbon}
\end{figure}
Such a knot is called ribbon. One easily sees that a ribbon disk can be pushed into the $4$--ball to produce a smooth disk, hence ribbon knots are slice. 
In \cite{Fox}, Fox asked whether the converse is true, {\em i.e.} if any slice knot is ribbon. 
This is nowadays known as the Slice-Ribbon conjecture. To determine whether a knot is ribbon (or slice) is a difficult task. 
One of the most famous obstruction is provided by the Fox-Milnor theorem on the Alexander polynomial \cite{FoxMilnor}: for a slice knot $K$, 
the Alexander polynomial $\Delta_K(t)\in\Zt$ can be written as $f(t)f(t^{-1})$ for some $f(t)\in\Zt$. 
This is what we will call here the factorization property. It emphasizes that some topology is reflected in this algebraic invariant.

For $2$--knots, {\em i.e.} embeddings of a $2$--sphere $S^2$ into the $4$--sphere $S^4$, one can still define an Alexander polynomial, 
but it was proven by Kinoshita \cite{Kin} that the Alexander polynomial of a ribbon $2$--knot ---a $2$-knot that bounds an immersed $3$--ball with only 
ribbon disk singularities--- can be any polynomial $f(t)\in\Zt$ such that $f(1)=1$. In \cite{NN}, Nakanishi and Nishizawa gave a topological condition 
on a $2$--knot ensuring that its Alexander polynomial is symmetric, a property satisfied by the Alexander polynomial of any knot in $S^3$. 
In this paper, we investigate the topological properties of $2$--knots that imply the factorization property of their Alexander polynomial. 
This leads us to introduce an alternative notion of ribbon $2$--knots which at the same time encompasses the usual notion of ribbon $2$--knots and 
is conveniently featured to recover for some subclasses the factorization property of the Alexander polynomial.

For classical knots, the ribbon singularities of an immersed disk bounded by the knot are necessarily $1$--disks ---the only compact connected 
$1$-manifolds with non-empty boundary. For $2$--knots, there are much more possibilities. Roseman proposed a general definition of ribbon $2$--knots 
with no condition on the topological type of the ribbon singularities (see Hitt \cite{Hitphd}). 
Here we focus on the case of ribbon singularities of annular type; we call $A$--ribbon a $2$--knot that bounds an immersed $3$--ball 
with only ribbon singularities of annular type. Beyond the fact that it is the next easiest possibility after $2$-disks, they appear naturally 
via Artin's spinning construction. This construction produces a $2$--knot from a classical knot, rotating it around an $\R^2$--axis. 
It has the property to preserve the Alexander polynomial, hence the spuns of ribbon knots are $2$--knots whose Alexander polynomial 
has the factorization property. It is easily seen that spinning the initial immersed $2$--disk produces an immersed $3$--ball whose 
self-intersections are annular ribbon singularities. The factorization property naturally arises for another class of $2$--knots, namely the connected sum 
of a $2$-knot with its mirror image. When the initial $2$--knot is ribbon, in the usual sense, there is a natural construction of an immersed 
$3$--ball bounded by this connected sum which has only annular ribbon singularities.

\begin{figure}
\begin{center}
\scalebox{0.7}{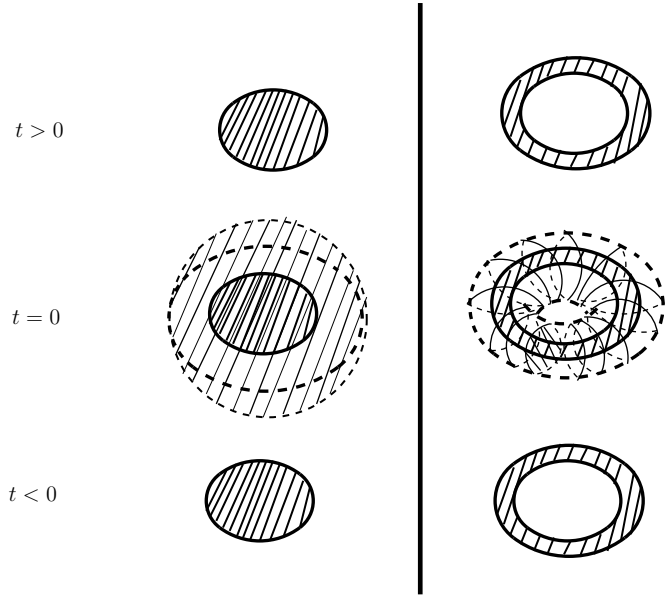}
\caption{Local models.} \label{figlocalmodels}
\end{center}
\end{figure}

This notion of an $A$--ribbon $2$--knot appears to generalize the usual notion of a ribbon $2$--knot.
\begin{prop}[Proposition \ref{propribbonAribbon}]
Ribbon $2$--knots are $A$--ribbon.
\end{prop}
In view of this proposition, the above mentioned result of Kinoshita implies that any polynomial $f(t)\in\Zt$ such that $f(1)=1$ 
is the Alexander polynomial of an A--ribbon $2$--knot. Hence we need to add some condition to recover the factorization property. 
Such a condition is defined in Subsection \ref{subseclink} and called the linkings condition, which concerns the relative positions 
of the preimages of the singularities in the preimage of the immersed $3$--ball. This condition is naturally satisfied by the spuns of 
ribbon knots.

\begin{theo}[Corollary \ref{corfacto}]
The Alexander polynomial of an $A$--ribbon $2$--knot satisfying the linkings condition has the factorization property.
\end{theo}

Alexander invariants can be defined in greater generality for every embedding $K$ of an $n$--sphere $S^n$ into the $(n+2)$--sphere $S^{n+2}$ ---such an embedding is called an $n$--knot. For an $n$--knot, there are $n$ Alexander polynomials, denoted $\Delta^{(k)}_K$ for $k$ from $1$ to $n$. 
Levine proved that they satisfy a remarkable property: $\Delta_K^{(k)}(t)=\Delta_K^{(n+1-k)}(t^{-1})$ for all $k=1,\ldots,n$. 
In particular, for a $2$--knot, ``the'' Alexander polynomial is $\Delta_K(t)=\Delta_K^{(1)}(t)=\Delta_K^{(2)}(t^{-1})$. 
In addition, if $n$ is odd and $K$ bounds a smooth ball $B^{n+1}$ in $B^{n+3}$ (the $n$--knot is still said to be slice), then the ``middle'' Alexander polynomial $\Delta^{\frac{n+1}{2}}(t)$ has the factorization property (Shinohara-Sumners \cite{ShSu}). For the other Alexander polynomials, the question 
arises to know what kind of topological properties would ensure the factorization property, discussed in this paper when $n=2$. A possible generalization 
of annular ribbon singularities is given by ribbon singularities that are products of circles $S^1$ with compact $(n-1)$--manifolds with 
non-empty boundaries. 


\subsection*{Outline of the paper.} In the next section, we give the definition and a characterization of $A$--ribbon $2$--knots and we relate this notion to the usual one. 
In Section \ref{secseifert}, we discuss the construction of Seifert hypersurfaces and the computation of Seifert matrices for $A$--ribbon $2$--knots. 
In the last section, we prove the factorization property for $A$-ribbon $2$--knots under the linkings condition.

\subsection*{Conventions and notations.}
The boundary of an oriented manifold with boundary is oriented with the ``outward normal first'' convention. 
We also use this convention to define the co-orientation of an oriented manifold embedded in another oriented manifold. \\ [0.05cm]
Given an oriented manifold $M$, we denote by $-M$ the same manifold with opposite orientation. \\ [0.05cm]
If $U$ and $V$ are transverse integral chains in a manifold $M$ such that $\dim(U)+\dim(V)=\dim(M)$, define the sign $\sigma_x$ of an intersection point $x\in U\cap V$ in the following way. Construct a basis of the tangent space $T_xM$ of $M$ at $x$ by taking an oriented basis of the normal space $N_xU$ followed by an oriented basis of $N_xV$. Set $\sigma_x=1$ if this basis is an oriented basis of $T_xM$ and $\sigma_x=-1$ otherwise. Now the algebraic intersection number of $U$ and $V$ in $M$ is $\langle U,V\rangle_M=\sum_{x\in U\cap V}\sigma_x$. \\ [0.05cm]
If $M$ and $N$ are submanifolds of dimension $k$ and $\ell$ respectively in an $n$-sphere $S^n$, with $k+\ell=n-1$, the linking number of $M$ and $N$ is $\lk(M,N)=\langle \Sigma,N\rangle_{S^n}$, where $\Sigma$ is a submanifold of $S^n$ such that $\partial \Sigma=M$. \\ [0.05cm]
The homology class of a submanifold $N$ in a manifold is denoted by $[N]$. \\ [0.05cm]
All homology groups are considered with coefficients in $\Z$.

\section{A--ribbon 2--knots and A--fusion 2--knots}

\begin{definition}
 Let $\Bpre\looparrowright\Brib\subset S^4$ be an immersion of a 3--ball $\Bpre$ in $S^4$. A {\em singularity} of $\Brib$ is a connected component 
 of the singular set of $\Brib$. A {\em pre-singularity} is a connected component of the preimage of this singular set. A singularity $R$ of $\Brib$ is {\em G--ribbon} if it contains only double points and its preimage in $\Bpre$ is made of:
 \begin{itemize}
  \item a {\em boundary pre-singularity} denoted $R^\partial$ properly embedded in $\Bpre$, meaning that 
   $R^\partial\cap\partial\Bpre=\partial R^\partial$,
  \item an {\em interior pre-singularity} denoted $R^\circ$ embedded in the interior of $\Bpre$.
 \end{itemize}
 A G--ribbon singularity is {\em ribbon} (resp. {\em A--ribbon}) if it is homeomorphic to a 2--disk (resp. to an annulus). 
\end{definition}

\begin{definition}
 A {\em ribbon ball} (resp. an {\em A--ribbon ball}) is an immersed 3--ball in $S^4$ whose singularities are ribbon (resp. A--ribbon). 
 A 2--knot is {\em ribbon} (resp. {\em A--ribbon}) if it bounds a ribbon ball (resp. an A--ribbon ball).
\end{definition}

\begin{proposition} \label{propribbonAribbon}
 Any ribbon 2--knot is A--ribbon.
\end{proposition}
\begin{proof}
 Let $K$ be a ribbon 2--knot bounding a ribbon ball $\Brib$. Let $h:\Bpre\looparrowright\Brib\subset S^4$ be an associated immersion. 
 Let $R$ be a singularity of $\Brib$. Take a path $\gamma$ in $\Bpre$ from $x\in\partial\Bpre$ to $y\in\Int(R^\partial)$ such that $x=\gamma\cap\partial\Bpre$ and $\gamma\setminus y$ is disjoint from all the pre-singularities. 
 Let $N(\gamma)$ be a regular neighborhood of $\gamma$ in $\Bpre$. Restricted to $N(\gamma)$, the immersion $h$ is injective. 
 Hence $h(N(\gamma))$ is a ball in $\Brib$ that meets $\partial B$ along a disk. Removing the interior of $h(N(\gamma))$ from $\Brib$ 
 corresponds to performing an isotopy on the knot $K$ and changes the ribbon singularity $R$ into an A--ribbon singularity. 
 \begin{figure}
 \begin{center}
 \scalebox{0.8}[0.6]{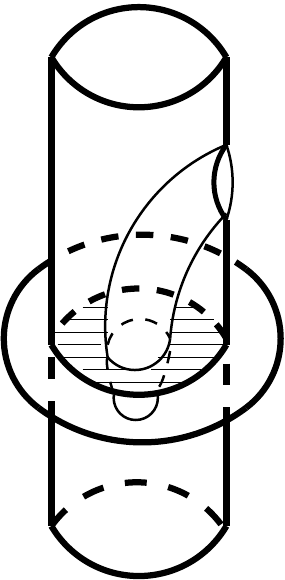}
 \caption{From a ribbon disk to a ribbon annulus.} \label{figfingermove}
 \end{center}
 \end{figure}
 Figure \ref{figfingermove} shows this finger move on the immersed ball $\Brib$ represented by a projection on a 3--dimensional hyperplane: from the local model of the left part of Figure \ref{figlocalmodels}, the 3--ball at $t=0$ has been projected onto a disk and the time direction has become the vertical direction.
\end{proof}

Let us define some notations. Given an annulus $R$ ---for instance a singularity or a pre-singularity, denote by $c(R)$ the core of $R$. 
Let $\Brib$ be an A--ribbon 3--ball and $\Bpre$ a preimage ball of $\Brib$. For a boundary pre-singularity $R^\partial$, define its co-core $\cc(R^\partial)$ as an arc on $R^\partial$ joining its two boundary components and transverse to $c(R^\partial)$. A closure $\ccc(R^\partial)$ of the co-core 
of $R^\partial$ is a knot in $\Bpre$ obtained from $\cc(R^\partial)$ by joining its endpoints with an arc embedded in $\partial\Bpre$. 
Note that the knot type of $\ccc(R^\partial)$ only depends on $R^\partial$. The pre-singularity $R^\partial$ divides $\Bpre$ into a ball denoted $B(R)$ and an integral homology torus denoted $T(R)$, with $R=B(R)\cap T(R)$. Notice that $T(R)$ is the exterior of the knot $\ccc(R^\partial)$. In particular, $T(R)$ is a standard torus if and only if $\ccc(R^\partial)$ is a trivial knot. 

We now introduce the fusion presentation of an A--ribbon 2--knot. 
Let $C_1,\dots,C_{k+1}$ be disjoint handlebodies trivially embedded in a 3--dimensional hyperplane in $S^4$. Let $E_1,\dots,E_k$ be disjoint copies of 
$S^1\times [0,1]\times[0,1]$ embedded in $S^4$ in such a way that:
\begin{itemize}
 \item $S^1\times[0,1]\times\{0\}$ and $S^1\times[0,1]\times\{1\}$ are embedded in the boundaries of the $C_i$'s, 
 \item $S^1\times[0,1]\times(0,1)$ is disjoint from the $\partial C_i$'s and meets transversely the interiors of the $C_i$'s along annuli,
 \item $\Brib=(\sqcup_{i=1}^{k+1}C_i)\cup(\sqcup_{i=1}^kE_i)$ is an immersed ball.
\end{itemize}
Such an immersed ball $\Brib$ is called an {\em A--fusion 3--ball}. It is immediate that the boundary of an A--fusion 3--ball is an A--ribbon 2--knot. 
We now prove the converse.
\begin{proposition}
 Any A--ribbon 2--knot bounds an A--fusion 3--ball. 
\end{proposition}
\begin{proof}
 Let $K$ be an A--ribbon 2--knot. Let $\Brib$ be an A--ribbon ball for $K$. We will modify the A--ribbon ball $\Brib$ in order to get an A--ribbon ball whose boundary pre-singularities are unknotted and unlinked, in the sense that the closures of their co-cores form a trivial link ---assuming that these closures are disjoint. 
 
 We first prove that we can split these co-core at any point. Fix a singularity $R$. Set $\xi=\ccc(R^\partial)$. Take a path $\gamma$ in $\Bpre$ from $x\in \Int(\xi\cap\partial\Bpre)$ to $y\in \Int(\xi\cap R^\partial)$ such that $x=\gamma\cap\partial\Bpre$ and the interior of $\gamma$ is disjoint from all the pre-singularities. Let $D$ be a disk embedded in $B(R)$ whose interior lies in the interior of $B(R)$, which is disjoint from all interior pre-singularities, whose intersections with other boundary pre-singularities, if any, are essential curves 
 on these pre-singularities, and such that $\partial D$ is an essential curve on $R^\partial$ containing $y$. Let $N$ be a neighborhood of $\gamma\cup D$. 
 Like in the proof of Proposition~\ref{propribbonAribbon}, remove the interior of $N$ to get a new A--ribbon ball for $K$, still denoted $\Brib$. 
 Note that the singularity $R$  gives rise to two singularities in the new A--ribbon ball. We now prove that this {\em cutting process} allows to unknot the boundary pre-singularity $R^\partial$. 
 
 Embed $\Bpre$ in $\R^3$ in such a way that there is a projection $p$ onto a plane such that $p(\xi\cap\partial\Bpre)\subset\partial p(\Bpre)$ and the singular points of $p_{|\xi}$ are transverse double points. Fix a crossing $c$ of $\xi$ in this projection. We will use the cutting process to change this crossing. Fix an orientation of $\xi$. Let $x$ be a point of $\xi\cap\partial\Bpre$ such that the arc of $\xi\cap\partial\Bpre$ going from $x$ to an endpoint of $\xi\cap\partial\Bpre$ does not meet any pre-singularity, except $R^\partial$ at its endpoint. 
 Let $y$ be a point of $\xi\cap\Int(R^\partial)$ such that $p(y)$ lies after the crossing $c$ and before the next crossing when running along $p(\xi)$ from $x$ in the sense of the orientation. Fix a framing of $\xi$ pointing toward $T(R)$ at any point. Take the arc of $\xi$ from $x$ to $y$ and push slightly its interior in the direction of the framing in order to define an arc $\gamma$ from $x$ to $y$ that satisfies the above requirements. 
 Add to $\gamma$ a turn around $B(R)$ in order to change its last crossing with $\xi$. 
 Apply the cutting process with this $\gamma$. The singularity $R$ is then divided into an unknotted singularity and a singularity whose closure of the co-core is a knot obtained from $\xi$ by changing the crossing $c$. Since any knot can be trivialized by crossing changes, this proves that we can unknot the boundary pre-singularity $R^\partial$. 
 
 We now unlink the link made of the closures of the co-cores of the boundary pre-singularities. Let $R_i$ for $1\leq i\leq n$ be the singularities of $\Brib$. Let $I\subset\{1,\dots,n\}$ be the minimal set satisfying $\cap_{i=1}^nT(R_i)=\cap_{i\in I}T(R_i)$. Use the cutting process to unknot the pre-singularities $R_i^\partial$ for $i\in I$. This turns the co-cores of these pre-singularities into a tangle in $B^3$. Such a tangle is always trivial up to isotopy. Anyway, we need to iterate by applying this procedure in each ball $B(R_i)$ for $i\in I$. Hence we have to consider the case where the components of the tangle have their extremities in two disjoint disks in $\partial\Bpre$. In this case, we have to prove that we can bring the tangle into a braid position. This can be done using the cutting process to change some crossings of the tangle. At each application of the cutting process, a new singularity appears that has both boundary components in the same disk. With an arc $\gamma$ that joins this new singularity to the other disk, we can apply once again the cutting process to turn the tangle component corresponding to the new singularity into two monotone components.
 
 So we can assume that the boundary pre-singularities of $\Bpre$ are unlinked. It follows that cutting $\Bpre$ along these boundary pre-singularities, we can write it as a disjoint union of handlebodies glued together by copies of $S^1\times [0,1]\times[0,1]$. Make these handlebodies trivially embedded in a common 3-dimensional hyperplane of $S^4$ by an ambient isotopy. This provides an A--fusion ball for $K$.
\end{proof}

\section{Seifert hypersurfaces and Seifert matrices} \label{secseifert}

  \subsection{Levine presentation of the Alexander module of a 2--knot}

We review here the presentation of the Alexander module given by a Seifert matrix. 
We first recall some definitions and well-known facts. Let $K$ be a 2--knot. Let $N(K)$ be a tubular neighborhood of $K$. Set 
$X=M\setminus \Int(T(K))$. Consider the projection $\pi : \pi_1(X) \to \frac{H_1(X)}{torsion} \cong \Z$ 
and the covering map $p : \tilde{X} \to X$ associated with its kernel. The automorphism group $\Aut(\tilde{X})$ of this covering is isomorphic to $\Z$ and acts on $H_1(\tilde{X})$. Denoting the action of a generator of $\Aut(\tilde{X})$ as the multiplication by $t$, 
we get a structure of $\Zt$--module on $\Al(K)=H_1(\tilde{X})$. This $\Zt$--module is called the \emph{Alexander module} of $K$. 
It is known to have a finite presentation, so that it has well-defined elementary ideals. The {\em Alexander polynomial} $\Delta_K$ of $K$ is 
the generator of the smallest principal ideal of $\Zt$ that contains the first elementary ideal of $\Al(K)$. It satisfies $\Delta_K(1)=1$.

Let $K$ be a 2--knot. Let $\Sigma$ be a Seifert hypersurface of $K$. 
Assume the homology groups of $\Sigma$ are torsion-free. Fix bases $(x_1,\dots,x_n)$ and $(X_1,\dots,X_n)$ of $H_1(\Sigma)$ and $H_2(\Sigma)$ respectively, given by homology classes of simple closed curves and embedded surfaces in $\Sigma$. For a simple closed curve $\gamma\subset\Sigma$, define $\gamma^+$ (resp. $\gamma^-$) as the push-off of $\gamma$ in the direction of the positive (resp. negative) normal of $\Sigma$. Define the positive and negative Seifert matrices of $K$ associated with $\Sigma$ and the above bases of its homology groups as:
$$V_\pm=\left(\lk(X_i,x_i^\pm)\right).$$
\begin{proposition}[Levine] \label{propLevine}
 If $H_1(\Sigma)$ is torsion-free, then the matrix $tV_+-V_-$ is a presentation matrix of the $\Zt$--module $\Al(K)$. In particular, $\Delta_K(t)=\det(tV_+-V_-)$. 
\end{proposition}

\begin{remark}
 In \cite[\S2]{L1}, Levine works over $\Q$ and gets a presentation of the $\Qt$--module $\Q\otimes\Al(K)$ for any 2--knot $K$. The only obstruction to work over the integers comes from the possible existence of torsion in the homology of the Seifert hypersurface and its complement. Note that Alexander duality and the universal coefficient theorem imply that $H_1(\Sigma)$ is torsion-free if and only if $H_1(S^4\setminus\Sigma)$ is torsion-free, while $H_2(\Sigma)$ and $H_2(S^4\setminus\Sigma)$ are always torsion-free thanks to Poincar\'e duality and the universal coefficient theorem.
\end{remark}

\begin{corollary} \label{corZtorsion}
 If a 2--knot admits a Seifert hypersurface $\Sigma$ such that $H_1(\Sigma)$ is torsion-free, then its Alexander module has no $\Z$--torsion.
\end{corollary}
\begin{proof}
 The matrix presentation of the Alexander module $\Al$ given by Proposition \ref{propLevine} is a square matrix $M$ whose determinant is the Alexander 
 polynomial $\Delta$. We have $\Al=\left(\oplus_{i=1}^n\Zt g_i\right)/\left(\oplus_{j=1}^n\Zt r_j\right)$. Take $a\in\Al$ and represent it 
 by a column vector expressing it in terms of the $g_i$. Assume $ka=0$ for some non trivial integer $k$. Then there is $b\in\oplus_{i=1}^n\Zt g_i$ 
 such that $ka=Mb$. Hence $k\mathrm{Cof}(M)a=\det(M)b=\Delta(t)b$, where $\mathrm{Cof}(M)$ is the cofactor matrix of $M$. Since $\Delta(1)=1$, it implies that $b=kc$ with $c\in\oplus_{i=1}^n\Zt g_i$, so that $a=Mc$ and finally $a=0$ in $\Al$.
\end{proof}

  \subsection{Seifert hypersurface associated with an A--ribbon ball}

In this subsection, we associate a hypersurface with any A--ribbon ball and we compute its homology. Under some condition, we deduce 
a presentation of the Alexander module of the 2--knot that bounds this A--ribbon ball.

Let $K$ be an A--ribbon 2--knot and let $\Brib$ be an A--ribbon ball for $K$. We will construct from $\Brib$ a Seifert hypersurface for $K$. 
Let $R$ be an A--ribbon singularity of $\Brib$. Let $h:\Bpre\looparrowright \Brib\subset S^4$ be an immersion associated with $\Brib$. 
Let $B_R^\partial$ (resp. $B_R^\circ$) be the image by $h$ of a regular neighborhood of $R^\partial$ (resp. of $R^\circ$) in $\Bpre$ 
that does not meet the other pre-singularities. We say that $B_R^\partial$ (resp. $B_R^\circ$) is the {\em boundary leaf} 
(resp. the {\em interior leaf}) of $\Brib$ at $R$. Let $N(R)$ be a regular neighborhood of $R$ in $S^4$ such that 
$N(R)\cap\Brib\subset B_R^\partial\cup B_R^\circ$. Remove from $\Brib$ the interior of $N(R)$. The created boundary is made of a 
$\partial([0,1]^2)\times S^1$ on $B_R^\circ$ and two $[0,1]\times S^1$ on $B_R^\partial$, where the $S^1$ factors correspond to the core of $R$. 
Glue the last two along $\{0\}\times[0,1]\times S^1$ and $\{1\}\times[0,1]\times S^1$, 
choosing which $[0,1]\times S^1$ is glued to which $\{i\}\times[0,1]\times S^1$ in order to respect the orientation of the hypersurface. 
The process is described Figure \ref{figslice} at a point of the $S^1$ factor. 
Performing the same manipulation at each singularity of $\Brib$, we get the {\em Seifert hypersurface of $K$ associated with $\Brib$}, which we denote $\Sigma$. 

\newcommand{\neast}[1]{\draw[pattern = north east lines] #1;}
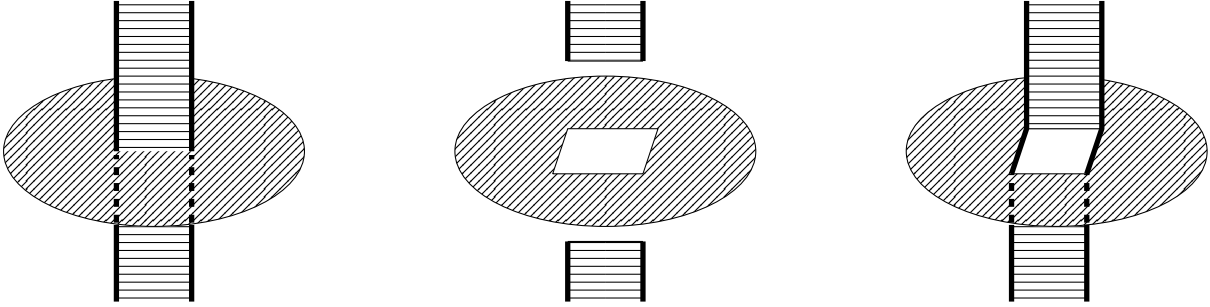
\begin{figure}[htb] 
\begin{center}
\begin{tikzpicture}
 \begin{scope}
  \hor{(0.5,-2) -- (0.5,0) -- (-0.5,0) -- (-0.5,-2)}
  \neup{(0,0) circle (2 and 1)}
  \horup{(-0.5,0) -- (-0.5,2) -- (0.5,2) -- (0.5,0)}
  \foreach \x in {-0.5,0.5} {
  \draw[line width=2pt] (\x,0) -- (\x,2);
  \draw[line width=2pt] (\x,-0.98) -- (\x,-2);
  \draw[line width=2pt,dashed] (\x,-0.95) -- (\x,-0.05);}
 \end{scope}
 \begin{scope} [xshift=6cm]
  \neast{(0,0) circle (2 and 1)}
  \draw[fill=white] (-0.7,-0.3) -- (-0.5,0.3) -- (0.7,0.3) -- (0.5,-0.3) -- (-0.7,-0.3);
  \hor{(0.5,2) -- (0.5,1.2) -- (-0.5,1.2) -- (-0.5,2)}
  \hor{(0.5,-2) -- (0.5,-1.2) -- (-0.5,-1.2) -- (-0.5,-2)}
  \draw (0.5,-1.2) -- (-0.5,-1.2) (0.5,1.2) -- (-0.5,1.2);
  \foreach \x in {-0.5,0.5} {
  \draw[line width=2pt] (\x,-2) -- (\x,-1.2) (\x,1.2) -- (\x,2);}
 \end{scope}
 \begin{scope} [xshift=12cm]
  \horup{(0.4,-2) -- (0.4,-0.3) -- (-0.6,-0.3) -- (-0.6,-2)}
  \neup{(0,0) circle (2 and 1)}
  \draw[fill=white] (-0.6,-0.3) -- (-0.4,0.3) -- (0.6,0.3) -- (0.4,-0.3) -- (-0.6,-0.3);
  \horup{(0.6,2) -- (0.6,0.3) -- (-0.4,0.3) -- (-0.4,2)}
  \draw (-0.4,0.3) -- (0.6,0.3);
  \foreach \x in {-0.6,0.4} {
  \draw[line width=2pt] (\x +0.2,2) -- (\x +0.2,0.3) -- (\x,-0.3);
  \draw[line width=2pt] (\x,-0.98) -- (\x,-2);
  \draw[line width=2pt,dashed] (\x,-0.95) -- (\x,-0.3);}
 \end{scope}
\end{tikzpicture}
\end{center}
\caption{A slice of a neighborhood of the singularity.} \label{figslice}
\end{figure}

We now have a closer look at the structure of the hypersurface $\Sigma$ and its homology groups. 
Denote by $R_i$ for $i=1,\dots,n$ the singularities of the A--ribbon ball $\Brib$. For each $i$, set $X_i=\partial B_{R_i}^\circ$. 
When we cut $\Sigma$ along the tori $X_i$, we see from the above construction that we obtain a 3--ball with $2n$ solid tori removed; 
denote it $\check{\Sigma}$ (see Figure \ref{figSeiferthyp}). Note that $\Sigma$ is recovered from $\check{\Sigma}$ by glueing $n$ handles $A_i$ homeomorphic to $S^1\times S^1\times I$, 
where the first $S^1$ factor corresponds to the core of $R_i$ in $\Brib$, the second $S^1$ factor corresponds to the meridian of $R_i$, 
and $X_i=S^1\times S^1\times\{1\}\subset A_i$.
\begin{figure}
\begin{center}
\scalebox{0.7}{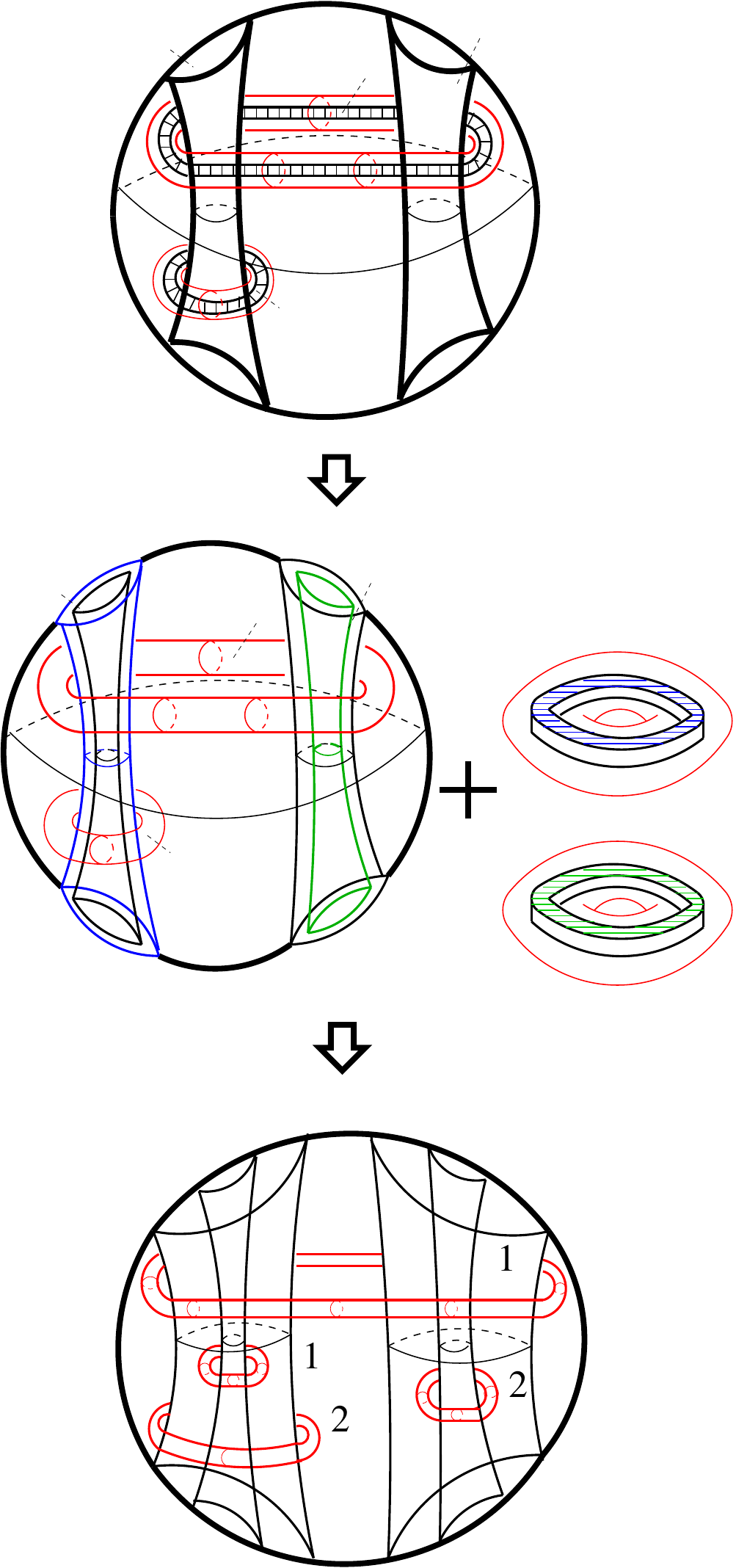}\hspace{1.5cm}\raisebox{1cm}{\scalebox{0.7}[0.8]{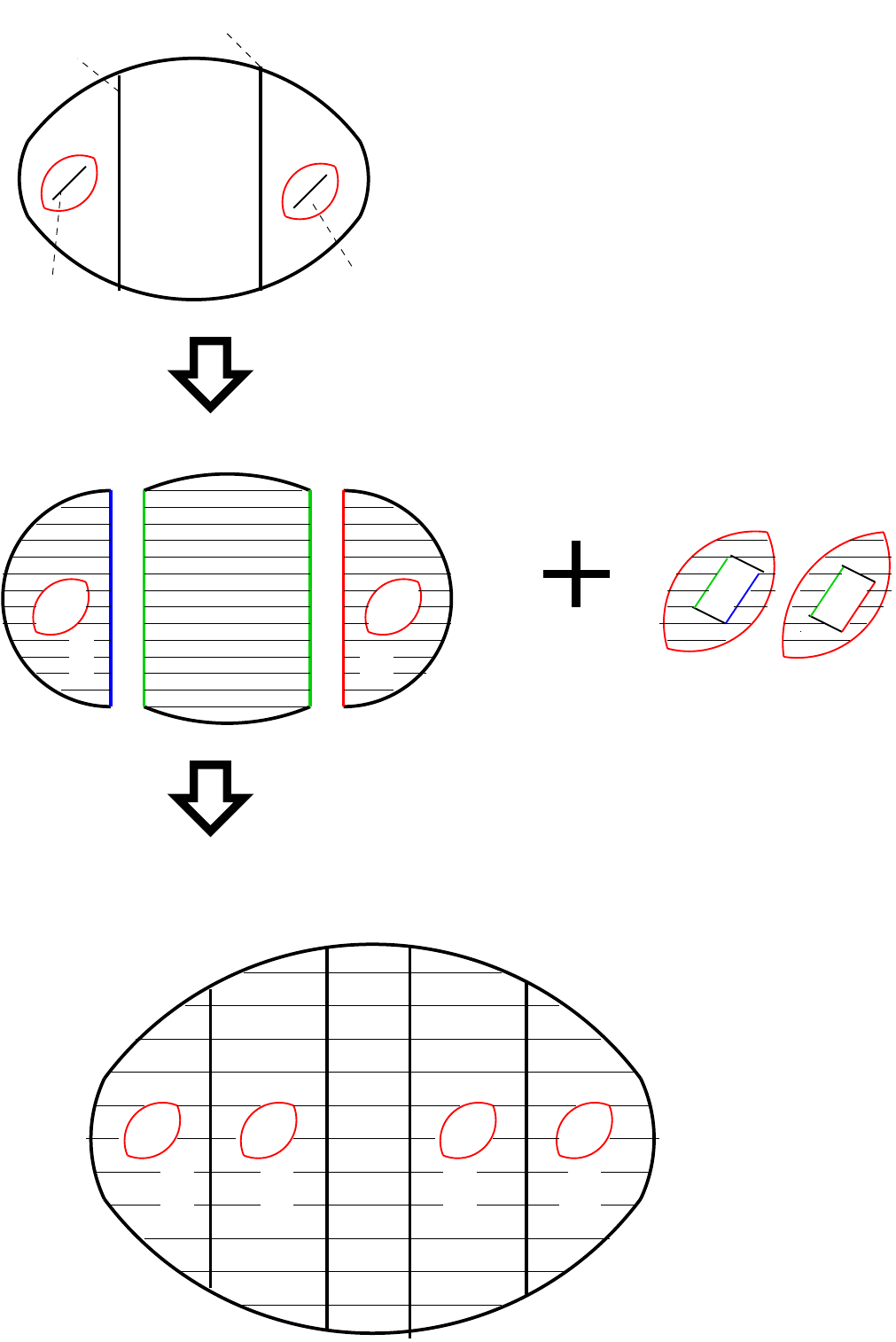}}
\caption{Seifert hypersurface associated with an A--ribbon ball.} \label{figSeiferthyp}
\end{center}
\end{figure}
For each $i$, set:
\begin{itemize}
 \item $x_i=\{*\}\times S^1\times\{1\}\subset A_i$, $x_i'=\{*\}\times S^1\times\{0\}\subset A_i$, $x_i^I=\{*\}\times S^1\times I\subset A_i$;
 \item $\beta_i=S^1\times\{*\}\times\{1\}\subset A_i$, $\beta_i'=S^1\times\{*\}\times\{0\}\subset A_i$, $\beta_i^I=S^1\times\{*\}\times I\subset A_i$;
 \item $X_i'=S^1\times S^1\times\{0\}\subset A_i$.
\end{itemize}
Note that $x_i$ corresponds to a meridian of the annulus $R_i^\circ$ in $\Bpre$. Let $y_i$ be a simple closed curve in $\Sigma$
such that $y_i\cap(\Sigma\setminus\check{\Sigma})=\{*\}\times\{*\}\times I\subset A_i$. 

Define the {\em link of interior pre-singularities of $\Brib$} as $L_S(\Brib)=\sqcup_{i=1}^n c(R_i^\circ)$ and the associated linking matrix 
$\Lk_S(\Brib)=\left(\lk_\Bpre(c(R_i^\circ),c(R_j^\circ))\right)_{1\leq i,j\leq n}$. Note that $\Lk_S(\Brib)$ is also the linking matrix of the link 
$\sqcup_{i=1}^n\beta_i$ viewed in $\Bpre$. This matrix plays a crucial role in the computation of $H_1(\Sigma)$ and $H_2(\Sigma)$. 
We will use the long exact sequence in homology associated with the pair $(\Sigma,\check{\Sigma})$. We first compute 
the relative homology groups. By excision, we have $H_k(\Sigma,\check{\Sigma})\cong H_k(\sqcup_{i=0}^nA_i,\sqcup_{i=0}^n\partial A_i)$. 
Thus:
\begin{itemize}
 \item $H_3(\Sigma,\check{\Sigma})\cong\Z^n$ is generated by the fundamental classes $[A_i]$,
 \item $H_2(\Sigma,\check{\Sigma})\cong\Z^{2n}$ is generated by the $[x_i^I]$ and the $[\beta_i^I]$,
 \item $H_1(\Sigma,\check{\Sigma})\cong\Z^n$ is generated by the classes of the $y_i\cap A_i$.
\end{itemize}
The homology of $\check{\Sigma}$ is easily computed: $H_1(\check{\Sigma})\cong\Z^{2n}$ is generated by the $[x_i]$ and $[x_i']$ and 
$H_2(\check{\Sigma})\cong\Z^{2n}$ is generated by the $[X_i]$ and $[X_i']$.
The long exact sequence gives:
$$0\to H_3(\Sigma,\check{\Sigma})\to H_2(\check{\Sigma})\to H_2(\Sigma)\to H_2(\Sigma,\check{\Sigma})\to H_1(\check{\Sigma})\to H_1(\Sigma)
 \to H_1(\Sigma,\check{\Sigma})\to 0.$$
Since $H_1(\Sigma,\check{\Sigma})$ is free, the sequence splits at $H_1(\Sigma)$ and we have 
$H_1(\Sigma)\cong H_1(\Sigma,\check{\Sigma})\oplus\left(H_1(\check{\Sigma})/\partial_2(H_2(\Sigma,\check{\Sigma}))\right)$.
Now $\partial_2([x_i^I])=[x_i]-[x_i']$ and $\partial_2([\beta_i^I])=[\beta_i]-[\beta_i']$. In $\check{\Sigma}$, the $\beta_i'$ bound 
embedded disks and $[\beta_i]=\sum_{j=1}^n\lk_{\Bpre}(\beta_i,\beta_j)[x_j]$. Thus we get:
$$H_1(\Sigma)\cong\Z^n\oplus\left(\Z^n/\Lk_S(\Brib)\Z^n\right),$$
where the first factor is freely generated by the $[y_i]$ and the second factor is generated by the $[x_i]$.
Similarly, $H_2(\Sigma,\check{\Sigma})$ is free, thus we have 
$H_2(\Sigma)\cong \ker(\partial_2)\oplus\left(H_2(\check{\Sigma})/\partial_3(H_3(\Sigma,\check{\Sigma}))\right)$. 
The expression of $\partial_2$ given above shows that $\ker(\partial_2)\cong\Z^{n-s}$ where $s$ is the rank of $\Lk_S(\Brib)$. 
One easily deduces $H_2(\Sigma)\cong\Z^{2n-s}$. 

We now assume $\Lk_S(\Brib)=0$. In this case, there are embedded surfaces $Y_i$ in $\Sigma$ such that $Y_i\cap(\Sigma\setminus\check{\Sigma})=\beta_i^I$. Fix orientations of the $c(R_i^\circ)$ and orient the $x_i$ so that $\lk_\Bpre(x_i,c(R_i^\circ))=1$. Orient the $y_i$ and $Y_i$ so that:
$$\langle X_i,y_j\rangle_\Sigma=\delta_{ij}\quad\textrm{and}\quad\langle Y_i,x_j\rangle_\Sigma=-\delta_{ij}.$$
The families $(x_1,\dots,x_n,y_1,\dots,y_n)$ and $(X_1,\dots,X_n,Y_1,\dots,Y_n)$ are bases of $H_1(\Sigma)$ and $H_2(\Sigma)$ respectively, 
dual in the above sense.
It is easily checked that $\lk(X_i,x_j^\pm)=0$ for any $i,j$. Hence the Seifert matrices associated with $\Sigma$ and the above basis of its homology groups 
are:
$$V_\pm=\begin{pmatrix}
         0 & U_\pm \\ W_\pm & *
        \end{pmatrix},$$
where $U_\pm=\left(\lk(X_i,y_j^\pm)\right)_{1\leq i,j\leq n}$ and $W_\pm=\left(\lk(Y_i,x_j^\pm)\right)_{1\leq i,j\leq n}$. 
Unlike the case of classical knots, we don't have $W_\pm=\tr{U_\mp}$ in general, thus we don't get the factorization property. 
We give in the next section a topological characterization of the A--ribbon balls that provide the equalities $W_\pm=\tr{U_\mp}$.

  \subsection{Computing the Seifert matrices from the preimage ball} \label{subseccomputeSeifert}

Keeping the notations of the previous subsection and the condition $\Lk_S(\Brib)=0$, we now explain how to compute the matrices $U_\pm$ and $W_\pm$ from the preimage ball with some orientation information. This information is given for each singularity $R$ by an arrow at a point of $R^\partial$ that gives the direction of the negative normal to the interior leaf $B_R^\circ$. Orient the boundary pre-singularities so that these arrows 
give the direction of their positive normal in $\Bpre$. We assume that the cores of all the singularities are oriented and we orient 
the cores of the pre-singularities accordingly. 
To make our computation, we need to have a closer look at the local picture around a singularity. 

Curves and surfaces can be drawn in $\Bpre$ that correspond to the elements of the bases $(x_1,\dots,x_n,y_1,\dots,y_n)$ and 
$(X_1,\dots,X_n,Y_1,\dots,Y_n)$ of $H_1(\Sigma)$ and $H_2(\Sigma)$ defined above; we will use the same notations. In $\Bpre$, 
$X_i$ is the oriented boundary of a tubular neighborhood of $R_i^\circ$, $x_i$ is an oriented meridian of the core $c(R_i^\circ)$, 
$y_i$ is an arc joining the two points in the preimage of a point of $R_i$ and $Y_i$ is a surface whose boundary 
is $c(R_i^\partial)\sqcup(-c(R_i^\circ))$. Choose the surface $Y_i$ as a disjoint union $Y_j=Y_j^\partial\sqcup (-Y_j^\circ)$ 
where $\partial Y_j^\partial=c(R_j^\partial)$, $\partial Y_j^\circ=c(R_j^\circ)$ and $Y_j^\partial$ is a disk properly embedded in $B(R_j^\partial)$. 
Choose the arc $y_i$ so that it meets $R_i^\partial$ and $R_i\circ$ only at its endpoints. 
The next two results express the coefficients of the matrices $U_\pm$ and $W_\pm$ in terms of algebraic intersections in $\Bpre$. They are illustrated in Figures \ref{Xi}.

\begin{figure}
\begin{center}
\raisebox{1.8cm}{\scalebox{0.9}{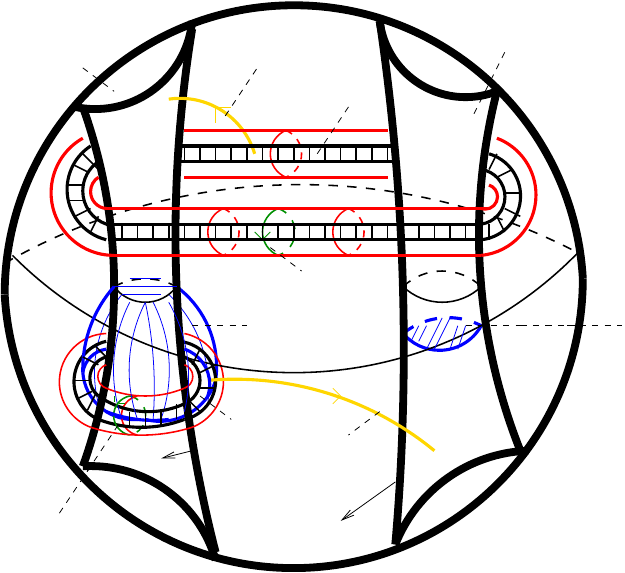}}\hspace{1.5cm}\scalebox{0.7}[0.8]{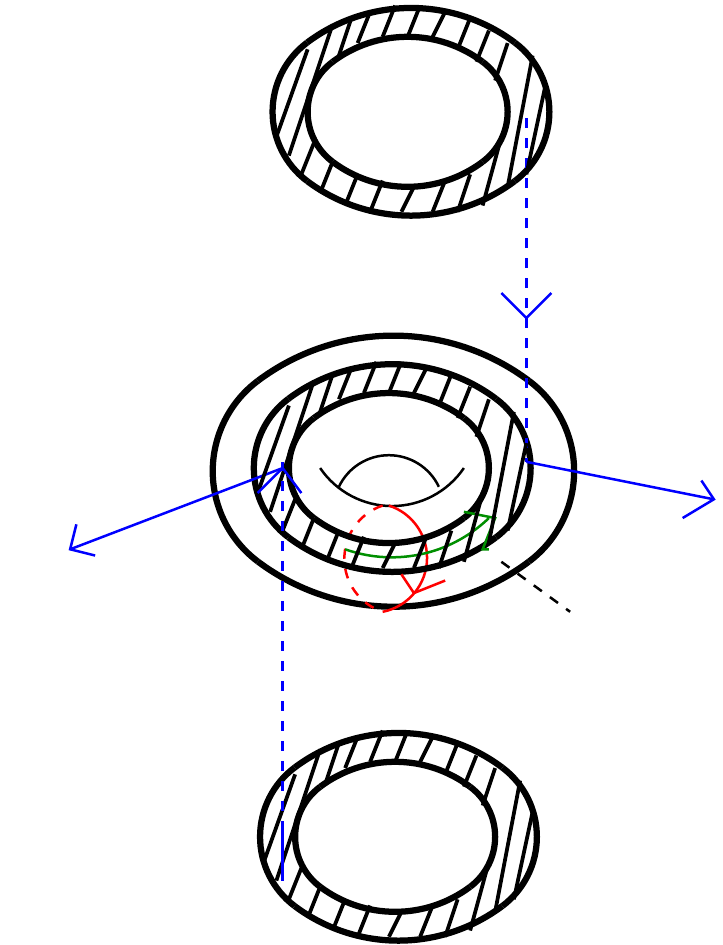}
\caption{Local computations of algebraic intersection numbers} \label{figcalculsigne}
\label{Xi}
\end{center}
\end{figure}

\begin{lemma} \label{lemmamatrixU}
 Set $\epsilon_i=1$ if the positive normal to $R_i^\partial$ gives the direction of $y_i$ at its endpoint on $R_i^\partial$, $\epsilon_i=-1$ otherwise. 
 We have:
 $$\lk(X_i,y_j^\pm)=\langle R_i^\partial,y_j\rangle_{\Bpre}\quad \textrm{ if } i\neq j,$$
 and: 
 $$\lk(X_i,y_i^+)=\left\lbrace\begin{array}{c} -1 \\ 0 \end{array}\right. \textrm{ and \ } 
   \lk(X_i,y_i^-)=\left\lbrace\begin{array}{c} 0 \\ 1 \end{array}\right.
   \begin{array}{l} \textrm{if } \epsilon_i=-1, \\ \textrm{otherwise.} \end{array}$$
\end{lemma}
\begin{proof}
 The torus $X_i^\pm$ is the boundary of a solid torus $T_i^\pm$ transverse to $\Sigma$, where $T_i=B_{R_i}^\circ$. 
 The linking $\lk(X_i^\pm,\gamma)$ of $X_i^\pm$ with a simple closed curve $\gamma$ transverse to $T_i^\pm$ is given by the algebraic 
 intersection number $\langle T_i^\pm,\gamma\rangle_{S^4}$. The contribution of an intersection point is $-1$ if $\gamma$ is oriented as the positive 
 normal to $T_i^\pm$ at that point, $1$ otherwise. If $\gamma=y_j$, such an intersection point corresponds in $\Bpre$ to an intersection point of $y_j$ 
 with the pre-singularity $R_i^\partial$. Hence if $j\neq i$:
 $$\lk(X_i,y_j^\pm)=\lk(X_i^\mp,y_j)=\langle R_i^\partial,y_j\rangle_{\Bpre}.$$
 In the case $j=i$, a special attention should be paid to the endpoints of the arc representing $y_i$ in $\Bpre$. These points correspond in $\Sigma$ 
 to an intersection of $y_i$ with either $T_i^+$ or $T_i^-$. It follows from the duality between $X_i$ and $y_i$ that the orientation of $y_i$ 
 at its endpoint lying on $R_i^\partial$ goes toward $R_i^\partial$. 
 If $y_i$ arrives on that point from the positive side (resp. negative side) of $R_i^\partial$, then the contribution of this point to 
 $\lk(X_i,y_i^+)$ is $-1$ (resp. $0$) and the contribution to $\lk(X_i,y_i^-)$ is $0$ (resp. $1$). 
 To conclude, note that $\langle R_i^\partial,\Int(y_i)\rangle_{\Bpre}=0$. 
\end{proof}

Orient the co-cores $\cc(R_i^\partial)$ of the boundary pre-singularities so that $(c(R_i^\partial),\cc(R_i^\partial))$ is an oriented basis 
of $R_i^\partial$. 
\begin{lemma} \label{lemmamatrixW}
 Set $\hat{\epsilon}_i=1$ if $\partial Y_i^\partial$ has a collar neighborhood in $Y_i^\partial$ that lies on the positive side of $R_i^\partial$, $\hat{\epsilon}_i=-1$ otherwise. 
 We have:
 $$\lk(Y_j,x_i^\pm)=\langle Y_j,\cc(R_i^\partial)\rangle_{\Bpre}\ \textrm{ if } i\neq j\quad \textrm{ and } \quad
   \lk(Y_i,x_i^\pm)=\langle -Y_i^\circ,\cc(R_i^\partial)\rangle_{\Bpre}+\rho_i^\pm,$$
 where: 
 $$\rho_i^+=\left\lbrace\begin{array}{c} 0 \\ 1 \end{array}\right. \textrm{ and \ } 
   \rho_i^-=\left\lbrace\begin{array}{c} -1 \\ 0 \end{array}\right.
   \begin{array}{l} \textrm{if } \hat{\epsilon}_i=1, \\ \textrm{otherwise.} \end{array}$$
\end{lemma}
\begin{proof}
 The curve $x_i$ is the boundary of a disk $D_i$ transverse to $\Sigma$ and isotopic to a meridian disk 
 of $B_{R_i}^\circ$. For an embedded surface $\Gamma$, disjoint from $x_i$ and transverse to $D_i$, we have 
 $\lk(\Gamma,x_i)=-\lk(x_i,\Gamma)=-\langle D_i,\Gamma\rangle_{S^4}$.
 If $\Gamma=Y_j^\pm$, an intersection point between $D_i$ and $\Gamma$ corresponds to an essential curve of $R_i^\partial$ in the intersection 
 $R_i^\partial\cap Y_j$. Checking the orientation conventions, one gets for $j\neq i$:
 $$\lk(Y_j,x_i^\pm)=\langle Y_j,\cc(R_i^\partial)\rangle_{\Bpre}.$$
 Once again, the case $j=i$ requires special attention for the boundary of $Y_i\partial$. This curve corresponds to an intersection point of the disk $D_i$ with either $Y_i^+$ or $Y_i^-$. By convention, the orientation of $Y_i^\partial$ is given near its boundary by first the direction pointing toward $T(R_i)$ and second the direction of $c(R_i^\partial)$. If the surface $Y_i^\partial$ lies on the positive side (resp. negative side) of $R_i^\partial$, 
 the contribution of this curve to $\lk(Y_i,x_i^+)$ is $0$ (resp. $1$) and its contribution to $\lk(Y_i,x_i^-)$ is $-1$ (resp. $0$). 
 Note that $\langle \Int(Y_i^\partial),\cc(R_i^\partial)\rangle_{\Bpre}=0$. 
\end{proof}

\begin{example} With the preimage ball on the left part of Figure \ref{figcalculsigne}, the associated Seifert matrices are
$V_+=\begin{pmatrix} 0 & \begin{array}{cc} 0 & 0 \\ 0 & -1 \end{array} \\
      \begin{array}{cc} 1 & -1 \\ 1 & 1 \end{array} & \star \end{pmatrix}$
and 
$V_+=\begin{pmatrix} 0 & \begin{array}{cc} 1 & 0 \\ 0 & 0 \end{array} \\
      \begin{array}{cc} 0 & -1 \\ 1 & 0 \end{array} & \star \end{pmatrix}$.
\end{example}

\section{Factorization of the Alexander polynomial}

In this section, we introduce some conditions on A--ribbon 2--knots that ensure the factorization property of the Alexander polynomial. 

  \subsection{Spun knots and concentricity}

Let us recall a construction of 2--knots from classical knots, first introduced by Artin \cite{Artin}. Consider a 3--dimensional half-space $\R^3_+\subset\R^4\subset S^4$. Let $\kappa$ be 
an arc embedded in $\R^3_+$ with endpoints in $\R^2=\partial\R^3_+$. Rotating around $\R^2$, this arc describes a 2--knot $K$ called the spun 
of the 1--knot $\bar{\kappa}$ obtained from $\kappa$ by joining its endpoints with an arc embedded in $\R^2$. Any Seifert matrix of $\bar{\kappa}$ 
is a Seifert matrix of $K$. In particular, $\bar{\kappa}$ and $K$ have the same Alexander polynomial. It follows that the spun of a ribbon 1--knot 
always has the factorization property. Moreover, if $\bar{\kappa}$ is a ribbon 1--knot, then its ribbon disk can also be rotated around $\R^2$, providing 
an A--ribbon 3--ball for $K$. Hence the spun of a ribbon 1--knot is an A--ribbon 2--knot.

We now introduce a condition on an A--ribbon 2--knot that ensures the factorization property of the Alexander polynomial and is satisfied in particular 
by the spuns of ribbon 1--knots. Let $\Brib$ be an A--ribbon 3--ball with singularities $R_i$, $i=1,\dots,n$. Let $\Bpre$ be a 3-ball preimage of $\Brib$. 
An {\em essential arc} in $\Bpre$ is an arc $\eta$ such that $\partial\eta=\eta\cap\partial \Bpre$ and $\eta$ is disjoint from the pre-singularities. 
Define the closure $\bar{\eta}$ of $\eta$ by joining its endpoints with an arc embedded in $\partial\Bpre$. 
For a pre-singularity $R_i^\star$ in $\Bpre$ whose core is oriented, where $\star$ stands for $\circ$ or $\partial$, define the {\em linking number} $\lk(R_i^\star,\eta)$ as the linking number in $\Bpre$ of the core of $R_i^\star$ with $\bar{\eta}$. 
Given an orientation of the cores of the singularities $R_i$, fix the corresponding orientations for the cores of the $R_i^\star$. 
The A--ribbon 3--ball $\Brib$ satisfies the {\em concentricity condition} if there is an orientation of the cores 
of the $R_i$ and an essential arc $\eta$ such that $\lk(R_i^\star,\eta)=1$ for all $i$ and all $\star$, and if the linking matrix of the pre-singularities is trivial. 
An A--ribbon 2--knot satisfies the concentricity condition if it bounds an A--ribbon 3--ball that satisfies it. 
\begin{lemma}
 The spun of a ribbon 1--knot satisfies the concentricity condition.
\end{lemma}
\begin{proof}
 Let $K$ be the spun of a ribbon 1--knot $\bar{\kappa}$ as in the above definition. Take a ribbon disk $\Drib$ of $\bar{\kappa}$. 
 Let $\Brib$ be the A--ribbon 3--ball obtained from $\Drib$. Define the essential arc $\eta$ as the arc used in the definition of spun knots to define $\bar{\kappa}$ from $\kappa$. The first part of the concentricity condition is easily checked. Now, for each singularity of $\Drib$, join the singularity to the arc $\eta$ by a path through its interior leaf; choose all these paths disjoint and with interiors disjoint from the singularities. Spinning these paths provides disjoint embedded disks in $\Brib$ bounded by the cores of the singularities, meeting them along their interior leaves. These disks lift in the preimage ball $\Bpre$ as disjoint 
 embedded disks bounded by the cores of the interior pre-singularities. 
\end{proof}

We will see in the next subsection that the concentricity condition implies the factorization property.

  \subsection{Singularities position} \label{subseclink}

We now introduce the characterization announced in Section \ref{secseifert}.
We begin with some definitions. Let $\Brib$ be an A--ribbon 3--ball with singularities $(R_i)_{1\leq i\leq n}$. Let $\Bpre$ be a preimage of $\Brib$. 
Orient the cores of the singularities $R_i$ and fix the corresponding orientations for the cores of the $R_i^\star$. 
Fix a boundary pre-singularity $R_i^\partial$. Define the {\em linking of the pre-singularity $R_j^\star$ with respect to $R_i^\partial$} as:
$$\ell_i(R_j^\star)=\left\lbrace\begin{array}{c l}
                  1 & \textrm{if } R_j^\star\subset B(R_i) \\
                  k & \textrm{if } R_j^\star\subset T(R_i) \textrm{ and } c(R_j^\star)=k\,c(R_i^\partial) \textrm{ in } H_1(T(R_i))
                  \end{array}\right..$$

The A--ribbon ball $\Brib$ with oriented singularities satisfies the {\em linkings condition} if $$\ell_i(R_j^\partial)=\ell_i(R_j^\circ)$$ for all $i,j$ 
and if the linking matrix of the pre-singularities is trivial. 
By extension, we say that an A--ribbon 2--knot satisfies the linkings condition if it bounds an A--ribbon 3--ball that satisfies this condition for given 
orientations of its singularities. Note that the concentricity condition implies the linkings condition. 

\begin{proposition} \label{proplinkcond}
 Let $\Brib$ be an A--ribbon 3--ball with oriented singularities $(R_i)_{1\leq i\leq n}$. Assume the linking matrix of the pre-singularities is trivial. 
 Let $\Sigma$ be the associated Seifert hypersurface and let $V_\pm=\begin{pmatrix} 0 & U_\pm \\ W_\pm & * \end{pmatrix}$ be associated Seifert matrices. Then $\Brib$ satisfies the linkings condition 
 if and only if $W_\pm=\tr{U_\mp}$.
\end{proposition}
\begin{corollary}\label{corfacto}
 If an A--ribbon 2--knot satisfies the linkings condition, then it has the factorization property.
\end{corollary}

Recall the surface $Y_j$ was defined as $Y_j=Y_j^\partial\sqcup (-Y_j^\circ)$ with $\partial Y_j^\partial=c(R_j^\partial)$ and
$\partial Y_j^\circ=c(R_j^\circ)$. We have 
$\langle Y_j,\cc(R_i^\partial)\rangle=\langle Y_j^\partial,\cc(R_i^\partial)\rangle-\langle Y_j^\circ,\cc(R_i^\partial)\rangle$ 
and the next lemma gives the two terms in terms of the linkings $\ell_i$ and the $\varepsilon_i$ defined by $\varepsilon_i=-1$ if the positive 
normal to $R_i^\partial$ points toward $B(R_i)$ and $\varepsilon_i=1$ otherwise.
\begin{lemma} \label{lemmaYcci}
 For $Y_j^\star\neq Y_i^\partial$:
 $$\langle Y_j^\star,\cc(R_i^\partial)\rangle_\Bpre=\left\lbrace\begin{array}{c l}
  0 & \textrm{if } R_j^\star\subset B(R_i) \\
  \varepsilon_i\,\ell_i(R_j^\star) & \textrm{if } R_j^\star\subset T(R_i)\phantom{\textrm{\Large{R}}}
 \end{array}\right.$$
\end{lemma}
\begin{proof}
 This algebraic intersection only depends on the boundary of $Y_j^\star$. If $R_j^\star\subset B(R_i)$, the surface $Y_j^\star$ can be chosen 
 in the interior of $B(R_i)$ and the result follows. Assume $R_j^\star\subset T(R_i)$. Consider a disk $D\subset\Bpre$ that intersects 
 $R_i^\partial$ transversely along a single simple closed curve isotopic to $c(R_i^\partial)$, whose oriented boundary is a push-off of $c(R_i^\partial)$ 
 in the direction of $T(R_i)$. We have $\langle Y_j^\star,\cc(R_i^\partial)\rangle=\ell_i(R_j^\star)\langle D,\cc(R_i^\partial)\rangle$.
\end{proof}

\begin{lemma} \label{lemmaRy}
For $i\neq j$:
 $$\langle R_i^\partial,y_j\rangle_\Bpre=\left\lbrace
                      \begin{array}{c l}
                       0 & \textrm{if } R_j^\partial,R_j^\circ\subset B(R_i) \textrm{ or } R_j^\partial,R_j^\circ\subset T(R_i) \\
                       \varepsilon_i & \textrm{if } R_j^\circ\subset B(R_i) \textrm{ and } R_j^\partial\subset T(R_i)\phantom{\textrm{\Large{R}}} \\
                       -\varepsilon_i & \textrm{if } R_j^\partial\subset B(R_i) \textrm{ and } R_j^\circ\subset T(R_i)\phantom{\textrm{\Large{R}}}
                      \end{array}\right.$$
\end{lemma}
\begin{proof}
 The choice of orientation for $y_j$ implies that it is oriented from $R_j^\circ$ to $R_j^\partial$. 
\end{proof}

\begin{proof}[Proof of Proposition \ref{proplinkcond}.]
 For any $i\neq j$, Lemmas \ref{lemmaYcci} and \ref{lemmaRy} imply that $\ell_i(R_j^\partial)=\ell_i(R_j^\circ)$ if and only if 
 $\langle R_i^\partial,y_j\rangle_{\Bpre}=\langle Y_j,\cc(R_i^\partial)\rangle_{\Bpre}$, using the fact that $\ell_i(R_j^\star)=1$ if $R_j^\star\subset B(R_i)$. This concludes thanks to Lemmas \ref{lemmamatrixU} 
 and \ref{lemmamatrixW}. If $i=j$, notice that $\ell_i(R_i^\partial)=1$ and conclude using the same lemmas. 
\end{proof}

  \subsection{Connected sum of a 2--knot with its mirror image}

Denote by $\mir{K}$ the mirror image of a 2--knot $K$ and by $K_1\sharp K_2$ the connected sum of two 2--knots $K_1$ and $K_2$.
\begin{proposition}
 For any 2--knot $K$, the 2--knot $K\sharp\mir{K}$ has the factorization property.
\end{proposition}
This result follows from the next lemma. 

\begin{lemma}\ 
 \begin{itemize}
  \item $\Delta_{\mir{K}}(t)=\Delta_K(t^{-1})$
  \item $\Delta_{K_1\sharp K_2}=\Delta_{K_1}\Delta_{K_2}$
 \end{itemize}
\end{lemma}

The next result shows that not any A--ribbon 2--knot with the factorization property is a connected sum of an A--ribbon 2--knot with its 
mirror image, via the exemple of the spun of the ribbon knot $6_1$.
\begin{proposition}
 Let $K$ be the spun of the knot $6_1$. Then there is no ribbon 2--knot $J$ such that $K$ is isotopic to $J\sharp\mir{J}$.
\end{proposition}
\begin{proof}
 The 2--knot $K$, as the knot $6_1$, has its first elementary ideal principal and generated by its Alexander polynomial, namely $(2t-1)(2-t)$, 
 and has $\Zt$ as a second elementary ideal. 
 
 Assume $K$ is isotopic to $J\sharp\mir{J}$ for some ribbon 2--knot $J$. 
 Up to exchanging $J$ and $\mir{J}$, we have $\Delta_J(t)=(2t-1)$ and $\Delta_{\mir{J}}(t)=(2-t)$. 
 Let $V(t)$ be a square presentation matrix of the integral Alexander module of $J$, of size $n$. Then $V(t^{-1})$ 
 is a presentation matrix of $\Al(\mir{J})$ and $W(t)=\begin{pmatrix} V(t) & 0 \\ 0 & V(t^{-1}) \end{pmatrix}$ is a presentation matrix of $\Al(K)$. 
 The second elementary ideal of $K$ is generated by the minors of size $2n-1$ of $W(t)$. These minors have the following forms:
 $$\begin{pmatrix} V(t) & 0 \\ 0 & \star \end{pmatrix} \quad \begin{pmatrix} \star & 0 \\ 0 & V(t^{-1}) \end{pmatrix} \quad 
   \begin{pmatrix} \star & 0_n \\ 0 & \star \end{pmatrix} \quad \begin{pmatrix} \star & 0 \\ 0_n & \star \end{pmatrix},$$
 where $0_n$ is the trivial square matrix of size $n$. Hence these minors are mutiples of $2t-1$ or $2-t$. It follows that evaluation 
 at $t=-1$ sends the second elementary ideal of $K$ onto an ideal of $\Z$ contained in $3\Z$, which is a contradiction since this ideal is the whole $\Zt$.
\end{proof}

\newcommand{\lft}[1]{\draw[pattern = horizontal lines] #1;}
\newcommand{\up}[1]{\draw[color=white,fill=white] #1; \draw[pattern = north east lines] #1;}
\newcommand{\rght}[1]{\draw[color=white,fill=white] #1; \draw[pattern = horizontal lines] #1;}

\begin{figure}[htb] 
\begin{center}
\begin{tikzpicture} [scale=0.18]
 \begin{scope}
  \lft{(0,0) -- (0,10) -- (4,10) -- (4,0) -- (0,0)}
  \up{(2.5,-2.5) -- (7.5,2.5) -- (7.5,12.5) -- (2.5,7.5) -- (2.5,-2.5)}
  \rght{(6,0) -- (6,10) -- (10,10) -- (10,0) -- (6,0)}
  \draw (5,-5) node {$t=0$};
 \end{scope}
 \begin{scope} [xshift=15cm]
  \draw (4,10) -- (4,0);
  \draw (6,0) -- (6,10);
  \draw (5,-5) node {$0<t<1$};
 \end{scope}
 \begin{scope} [xshift=30cm]
  \lft{(6,0) -- (6,10) -- (4,10) -- (4,0) -- (6,0)}
  \draw (5,-5) node {$t=1$};
 \end{scope}
 \begin{scope} [xshift=45cm]
  \lft{(6,0) -- (6,10) -- (4,10) -- (4,0) -- (6,0)}
  \draw (5,-5) node {$t=3$};
 \end{scope}
 \begin{scope} [xshift=60cm]
  \draw (4,10) -- (4,0);
  \draw (6,0) -- (6,10);
  \draw (5,-5) node {$3<t<4$};
 \end{scope}
 \begin{scope} [xshift=75cm]
  \lft{(0,0) -- (0,10) -- (4,10) -- (4,0) -- (0,0)}
  \up{(2.5,-2.5) -- (7.5,2.5) -- (7.5,12.5) -- (2.5,7.5) -- (2.5,-2.5)}
  \rght{(6,0) -- (6,10) -- (10,10) -- (10,0) -- (6,0)}
  \draw (5,-5) node {$t=4$};
 \end{scope}
\end{tikzpicture}
\end{center}
\caption{The 2--knots $J$ and $\mir{J}$.} \label{figconnectsum}
\end{figure}
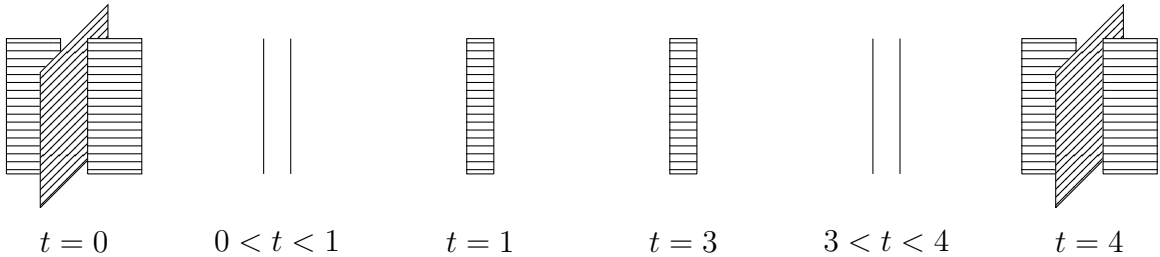

\begin{figure}[htb] 
\begin{center}
\begin{tikzpicture} [scale=0.18]
 \begin{scope}
  \lft{(0,0) -- (0,10) -- (4,10) -- (4,0) -- (0,0)}
  \up{(2.5,-2.5) -- (7.5,2.5) -- (7.5,12.5) -- (2.5,7.5) -- (2.5,-2.5)}
  \rght{(6,0) -- (6,10) -- (10,10) -- (10,0) -- (6,0)}
  \draw (5,-5) node {$0\leq t<1$};
 \end{scope}
 \begin{scope} [xshift=22cm]
  \lft{(0,0) -- (0,10) -- (5,10) -- (5,0) -- (0,0)}
  \up{(2.5,-2.5) -- (7.5,2.5) -- (7.5,12.5) -- (2.5,7.5) -- (2.5,-2.5)}
  \rght{(5,0) -- (5,10) -- (10,10) -- (10,0) -- (5,0)}
  \draw (5,-5) node {$1\leq t\leq3$};
 \end{scope}
 \begin{scope} [xshift=44cm]
  \lft{(0,0) -- (0,10) -- (4,10) -- (4,0) -- (0,0)}
  \up{(2.5,-2.5) -- (7.5,2.5) -- (7.5,12.5) -- (2.5,7.5) -- (2.5,-2.5)}
  \rght{(6,0) -- (6,10) -- (10,10) -- (10,0) -- (6,0)}
  \draw (5,-5) node {$3<t\leq4$};
 \end{scope}
\end{tikzpicture}
\end{center}
\caption{The immersed 3--ball.} \label{fig3ball}
\end{figure}
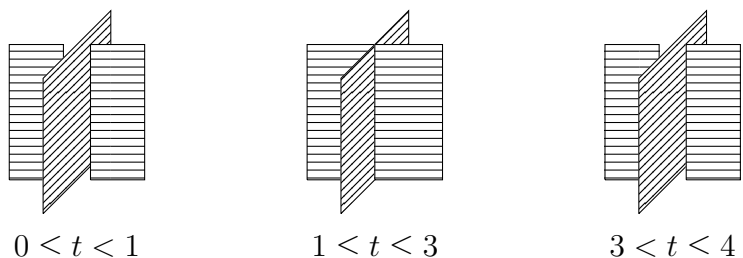

When $J$ is a ribbon 2--knot, it is clear that $K=J\sharp\mir{J}$ is also ribbon and it follows that it is A--ribbon. Anyway, it is interesting to note that there is a natural construction of an A--ribbon 3--ball for $K$ associated with the decomposition of $K$ as the connected sum $J\sharp\mir{J}$. 
It was proved by Yanagawa \cite{Yan} that a 2--knot is ribbon if and only if it is simply knotted, {\em i.e.} 
if it has a projection on a 3--dimensional hyperplane whose singular set is made of simple closed curves of double points. 
Consider such a projection of $J$ on $\R^3_0$ in $\R^4=\sqcup_{t\in\R}\R^3_t$. Re-construct $J$ from this projection by pushing the over-crossing leaf at each curve of double points as shown in Figure \ref{figconnectsum}, so that $J\subset\sqcup_{0\leq t\leq 1}\R^3_t$. Draw $\mir{J}$ by symmetry with respect to $\R^3_2$. Join each point of $J$ to the corresponding point of $\mir{J}$ by a line segment, see Figure \ref{fig3ball}. The union of all these line segments 
is an immersed $S^2\times S^1$; remove a tube from it, disjoint from the singularities, so that the obtained immersed 3--ball $\Brib$ is bounded by $J\sharp\mir{J}$. 
The singularities of $\Brib$ are easily seen to be ribbon annuli.

\bibliographystyle{smfalpha.bst}
\bibliography{biblio.bib}

\end{document}

%% file: local.pdf_t
\begin{picture}(0,0)%
\includegraphics{local.pdf}%
\end{picture}%
\setlength{\unitlength}{1579sp}%
\begingroup\makeatletter\ifx\SetFigFont\undefined%
\gdef\SetFigFont#1#2#3#4#5{%
  \reset@font\fontsize{#1}{#2pt}%
  \fontfamily{#3}\fontseries{#4}\fontshape{#5}%
  \selectfont}%
\fi\endgroup%
\begin{picture}(14738,13447)(-2039,-24522)
\put(-2024,-22336){\makebox(0,0)[lb]{\smash{{\SetFigFont{12}{14.4}{\familydefault}{\mddefault}{\updefault}{\color[rgb]{0,0,0}$t<0$}%
}}}}
\put(-1874,-14161){\makebox(0,0)[lb]{\smash{{\SetFigFont{12}{14.4}{\familydefault}{\mddefault}{\updefault}$t>0$}}}}
\put(-1949,-18361){\makebox(0,0)[lb]{\smash{{\SetFigFont{12}{14.4}{\familydefault}{\mddefault}{\updefault}$t=0$}}}}
\end{picture}%

%% file: finger.pdf_t
\begin{picture}(0,0)%
\includegraphics{finger.pdf}%
\end{picture}%
\setlength{\unitlength}{1579sp}%
\begingroup\makeatletter\ifx\SetFigFont\undefined%
\gdef\SetFigFont#1#2#3#4#5{%
  \reset@font\fontsize{#1}{#2pt}%
  \fontfamily{#3}\fontseries{#4}\fontshape{#5}%
  \selectfont}%
\fi\endgroup%
\begin{picture}(3415,6920)(8833,-25232)
\end{picture}%

%% file: ball.pdf_t
\begin{picture}(0,0)%
\includegraphics{ball.pdf}%
\end{picture}%
\setlength{\unitlength}{1973sp}%
\begingroup\makeatletter\ifx\SetFigFont\undefined%
\gdef\SetFigFont#1#2#3#4#5{%
  \reset@font\fontsize{#1}{#2pt}%
  \fontfamily{#3}\fontseries{#4}\fontshape{#5}%
  \selectfont}%
\fi\endgroup%
\begin{picture}(9765,20496)(397,-22984)
\put(4051,-6661){\makebox(0,0)[lb]{\smash{{\SetFigFont{14}{16.8}{\familydefault}{\mddefault}{\updefault}{\color[rgb]{0,0,0}2}%
}}}}
\put(6676,-3136){\makebox(0,0)[lb]{\smash{{\SetFigFont{14}{16.8}{\familydefault}{\mddefault}{\updefault}{\color[rgb]{0,0,0}2}%
}}}}
\put(2326,-3211){\makebox(0,0)[lb]{\smash{{\SetFigFont{14}{16.8}{\familydefault}{\mddefault}{\updefault}{\color[rgb]{0,0,0}1}%
}}}}
\put(5101,-3661){\makebox(0,0)[lb]{\smash{{\SetFigFont{14}{16.8}{\familydefault}{\mddefault}{\updefault}{\color[rgb]{0,0,0}1}%
}}}}
\put(9901,-11161){\makebox(0,0)[lb]{\smash{{\SetFigFont{14}{16.8}{\familydefault}{\mddefault}{\updefault}{\color[rgb]{0,0,0}1}%
}}}}
\put(9901,-13636){\makebox(0,0)[lb]{\smash{{\SetFigFont{14}{16.8}{\familydefault}{\mddefault}{\updefault}{\color[rgb]{0,0,0}2}%
}}}}
\put(2626,-13786){\makebox(0,0)[lb]{\smash{{\SetFigFont{14}{16.8}{\familydefault}{\mddefault}{\updefault}{\color[rgb]{0,0,0}2}%
}}}}
\put(5251,-10261){\makebox(0,0)[lb]{\smash{{\SetFigFont{14}{16.8}{\familydefault}{\mddefault}{\updefault}{\color[rgb]{0,0,0}2}%
}}}}
\put(901,-10336){\makebox(0,0)[lb]{\smash{{\SetFigFont{14}{16.8}{\familydefault}{\mddefault}{\updefault}{\color[rgb]{0,0,0}1}%
}}}}
\put(3676,-10786){\makebox(0,0)[lb]{\smash{{\SetFigFont{14}{16.8}{\familydefault}{\mddefault}{\updefault}{\color[rgb]{0,0,0}1}%
}}}}
\end{picture}%

%% file: Shema.pdf_t
\begin{picture}(0,0)%
\includegraphics{Shema.pdf}%
\end{picture}%
\setlength{\unitlength}{2368sp}%
\begingroup\makeatletter\ifx\SetFigFont\undefined%
\gdef\SetFigFont#1#2#3#4#5{%
  \reset@font\fontsize{#1}{#2pt}%
  \fontfamily{#3}\fontseries{#4}\fontshape{#5}%
  \selectfont}%
\fi\endgroup%
\begin{picture}(8061,12093)(2152,-14194)
\put(8401,-8536){\makebox(0,0)[lb]{\smash{{\SetFigFont{17}{20.4}{\familydefault}{\mddefault}{\updefault}{\color[rgb]{0,0,0}2}%
}}}}
\put(2776,-8161){\makebox(0,0)[lb]{\smash{{\SetFigFont{17}{20.4}{\familydefault}{\mddefault}{\updefault}{\color[rgb]{0,0,0}1}%
}}}}
\put(5401,-8236){\makebox(0,0)[lb]{\smash{{\SetFigFont{17}{20.4}{\familydefault}{\mddefault}{\updefault}{\color[rgb]{0,0,0}2}%
}}}}
\put(3226,-6286){\makebox(0,0)[lb]{\smash{{\SetFigFont{17}{20.4}{\familydefault}{\mddefault}{\updefault}{\color[rgb]{0,0,0}2}%
}}}}
\put(5026,-6286){\makebox(0,0)[lb]{\smash{{\SetFigFont{17}{20.4}{\familydefault}{\mddefault}{\updefault}{\color[rgb]{0,0,0}1}%
}}}}
\put(5401,-4636){\makebox(0,0)[lb]{\smash{{\SetFigFont{17}{20.4}{\familydefault}{\mddefault}{\updefault}{\color[rgb]{0,0,0}2}%
}}}}
\put(2401,-4861){\makebox(0,0)[lb]{\smash{{\SetFigFont{17}{20.4}{\familydefault}{\mddefault}{\updefault}{\color[rgb]{0,0,0}1}%
}}}}
\put(2551,-2611){\makebox(0,0)[lb]{\smash{{\SetFigFont{17}{20.4}{\familydefault}{\mddefault}{\updefault}{\color[rgb]{0,0,0}2}%
}}}}
\put(3901,-2461){\makebox(0,0)[lb]{\smash{{\SetFigFont{17}{20.4}{\familydefault}{\mddefault}{\updefault}{\color[rgb]{0,0,0}1}%
}}}}
\put(3601,-13036){\makebox(0,0)[lb]{\smash{{\SetFigFont{17}{20.4}{\familydefault}{\mddefault}{\updefault}{\color[rgb]{0,0,0}1}%
}}}}
\put(6226,-13036){\makebox(0,0)[lb]{\smash{{\SetFigFont{17}{20.4}{\familydefault}{\mddefault}{\updefault}{\color[rgb]{0,0,0}1}%
}}}}
\put(4501,-13036){\makebox(0,0)[lb]{\smash{{\SetFigFont{17}{20.4}{\familydefault}{\mddefault}{\updefault}{\color[rgb]{0,0,0}2}%
}}}}
\put(7276,-13036){\makebox(0,0)[lb]{\smash{{\SetFigFont{17}{20.4}{\familydefault}{\mddefault}{\updefault}{\color[rgb]{0,0,0}2}%
}}}}
\put(9526,-8536){\makebox(0,0)[lb]{\smash{{\SetFigFont{17}{20.4}{\familydefault}{\mddefault}{\updefault}{\color[rgb]{0,0,0}1}%
}}}}
\end{picture}%

%% file: Xi.pdf_t
\begin{picture}(0,0)%
\includegraphics{Xi.pdf}%
\end{picture}%
\setlength{\unitlength}{1973sp}%
\begingroup\makeatletter\ifx\SetFigFont\undefined%
\gdef\SetFigFont#1#2#3#4#5{%
  \reset@font\fontsize{#1}{#2pt}%
  \fontfamily{#3}\fontseries{#4}\fontshape{#5}%
  \selectfont}%
\fi\endgroup%
\begin{picture}(5985,5496)(1831,-7984)
\put(4726,-6886){\makebox(0,0)[lb]{\smash{{\SetFigFont{14}{16.8}{\familydefault}{\mddefault}{\updefault}{\color[rgb]{1,.84,0}$y_2$}%
}}}}
\put(4051,-6661){\makebox(0,0)[lb]{\smash{{\SetFigFont{14}{16.8}{\familydefault}{\mddefault}{\updefault}{\color[rgb]{0,0,0}2}%
}}}}
\put(6676,-3136){\makebox(0,0)[lb]{\smash{{\SetFigFont{14}{16.8}{\familydefault}{\mddefault}{\updefault}{\color[rgb]{0,0,0}2}%
}}}}
\put(2326,-3211){\makebox(0,0)[lb]{\smash{{\SetFigFont{14}{16.8}{\familydefault}{\mddefault}{\updefault}{\color[rgb]{0,0,0}1}%
}}}}
\put(5101,-3661){\makebox(0,0)[lb]{\smash{{\SetFigFont{14}{16.8}{\familydefault}{\mddefault}{\updefault}{\color[rgb]{0,0,0}1}%
}}}}
\put(7801,-5836){\makebox(0,0)[lb]{\smash{{\SetFigFont{14}{16.8}{\familydefault}{\mddefault}{\updefault}{\color[rgb]{0,0,1}$Y_2^{\partial}$}%
}}}}
\put(4201,-5911){\makebox(0,0)[lb]{\smash{{\SetFigFont{14}{16.8}{\familydefault}{\mddefault}{\updefault}{\color[rgb]{0,0,1}$Y_2^o$}%
}}}}
\put(4276,-3136){\makebox(0,0)[lb]{\smash{{\SetFigFont{14}{16.8}{\familydefault}{\mddefault}{\updefault}{\color[rgb]{1,.84,0}$y_1$}%
}}}}
\put(4651,-5461){\makebox(0,0)[lb]{\smash{{\SetFigFont{14}{16.8}{\familydefault}{\mddefault}{\updefault}{\color[rgb]{0,.56,0}$x_1$}%
}}}}
\put(2101,-7786){\makebox(0,0)[lb]{\smash{{\SetFigFont{14}{16.8}{\familydefault}{\mddefault}{\updefault}{\color[rgb]{0,.56,0}$x_2$}%
}}}}
\end{picture}%

%% file: Xi2.pdf_t
\begin{picture}(0,0)%
\includegraphics{Xi2.pdf}%
\end{picture}%
\setlength{\unitlength}{1579sp}%
\begingroup\makeatletter\ifx\SetFigFont\undefined%
\gdef\SetFigFont#1#2#3#4#5{%
  \reset@font\fontsize{#1}{#2pt}%
  \fontfamily{#3}\fontseries{#4}\fontshape{#5}%
  \selectfont}%
\fi\endgroup%
\begin{picture}(8633,11311)(5761,-23628)
\put(14026,-17986){\makebox(0,0)[lb]{\smash{{\SetFigFont{12}{14.4}{\familydefault}{\mddefault}{\updefault}{\color[rgb]{0,0,1}y'}%
}}}}
\put(5776,-13936){\makebox(0,0)[lb]{\smash{{\SetFigFont{12}{14.4}{\familydefault}{\mddefault}{\updefault}$t>0$}}}}
\put(5776,-18061){\makebox(0,0)[lb]{\smash{{\SetFigFont{12}{14.4}{\familydefault}{\mddefault}{\updefault}$t=0$}}}}
\put(5926,-22411){\makebox(0,0)[lb]{\smash{{\SetFigFont{12}{14.4}{\familydefault}{\mddefault}{\updefault}{\color[rgb]{0,0,0}$t<0$}%
}}}}
\put(12601,-19936){\makebox(0,0)[lb]{\smash{{\SetFigFont{12}{14.4}{\familydefault}{\mddefault}{\updefault}T}}}}
\put(9976,-20161){\makebox(0,0)[lb]{\smash{{\SetFigFont{12}{14.4}{\familydefault}{\mddefault}{\updefault}{\color[rgb]{1,0,0}x}%
}}}}
\put(12976,-17311){\makebox(0,0)[lb]{\smash{{\SetFigFont{12}{14.4}{\familydefault}{\mddefault}{\updefault}X}}}}
\put(7126,-19336){\makebox(0,0)[lb]{\smash{{\SetFigFont{12}{14.4}{\familydefault}{\mddefault}{\updefault}{\color[rgb]{0,0,1}y}%
}}}}
\end{picture}%